\documentclass[12pt]{amsart}
\usepackage{amssymb,amsmath,amsthm, enumerate}
\newtheorem{theorem}{Theorem}
\newtheorem*{theorem nonum}{Theorem}

\newtheorem{proposition}[theorem]{Proposition}
\newtheorem{corollary}[theorem]{Corollary}

\theoremstyle{remark}

\newtheorem*{remark}{Remark}

\numberwithin{theorem}{section} \numberwithin{equation}{section}

\providecommand{\abs}[1]{\lvert#1\rvert}
\DeclareRobustCommand{\stirling}{\genfrac\{\}{0pt}{}}

\begin{document}

\hspace{0.2in}
\title{Tensor Powers of the Defining Representation of $S_n$}
\author{Shanshan Ding}
\email{dish@sas.upenn.edu}
\begin{abstract}We give a decomposition formula for tensor powers of the defining representation of $S_n$ and apply it to bound the mixing time of a Markov chain on $S_n$.\end{abstract}
\maketitle

\section{Introduction}

The defining, or permutation, representation of $S_n$ is the $n$-dimensional representation $\varrho$ where
\begin{equation}
  \ (\varrho(\sigma))_{i, j} = \begin{cases}
		1  &\text{$\sigma(j)=i$}\\
		0 			 &\text{otherwise}.
    \end{cases}
\end{equation}
Since the fixed points of $\sigma$ can be read off of the matrix diagonal, the character of $\varrho$ at $\sigma$, $\chi_{\varrho}(\sigma)$, is precisely the number of fixed points of $\sigma$.  The irreducible representations, or irreps for short, of $S_n$ are parametrized by the partitions of $n$, and $\varrho$ decomposes as $S^{(n-1, 1)} \oplus S^{(n)}$.  Note that $\chi_{S^{(n-1, 1)}}(\sigma)$ is one less than the number of fixed points of $\sigma$.  In the terminology of \cite{FH91}, we call the $(n-1)$-dimensional irrep $S^{(n-1, 1)}$ the standard representation of $S_n$.

A classic question in the representation theory of symmetric groups is how tensor products of representations decompose as direct sums of irreps.  In Section 2 we will present a neat formula for the decomposition of tensor powers of $\varrho$ and, as corollary, that of tensor powers of $S^{(n-1, 1)}$.

Our study of tensor powers of $\varrho$ arose from an investigation in the mixing time of the Markov chain on $S_n$ formed by applying a single uniformly chosen $n$-cycle to a deck of $n$ cards and following up with repeated random transpositions.  This chain is a natural counterpart to the random transposition walk on $S_n$, famously shown by Diaconis and Shahshahani in \cite{DS81} to mix in $O(n \ln n)$ steps, in the sense that random transpositions induce Markov chains on not just $S_n$, but the set of partitions of $n$: the time-homogeneous random transposition walk is one such chain that starts at the partition $(1^n)$, whereas the process we proposed is one that starts at the other extreme, $(n)$.  Along with following the classic approach of \cite{DS81}, we will use the tensor decomposition formula to show in Section 3 that the mixing time for the $n$-cycle-to-transpositions chain is $O(n)$.

\section{Decomposition Formula for Tensor Powers of $\varrho$}

Let $\lambda$ be a partition of $n$, and recall that the irreps of $S_n$, the $S^{\lambda}$'s, are indexed by the partitions of $n$.  As promised, we give a compact formula for the decomposition of tensor powers of $\varrho$ into irreps, i.e. the coefficients $a_{\lambda, r}$ in the expression 
\begin{equation}
\varrho^{\otimes r} = \underset{\lambda \vdash n}{\bigoplus} a_{\lambda, r} S^{\lambda} := \underset{\lambda \vdash n}{\bigoplus} (S^{\lambda})^{\oplus a_{\lambda, r}}.
\end{equation}

\begin{proposition}\label{PrIrreps}
Let $\lambda \vdash n$ and $1 \le r \le n-\lambda_2$.  The multiplicity of $S^{\lambda}$ in the irreducible representation decomposition of $\varrho^{\otimes r}$ is given by
\begin{equation}\label{EqIrreps}
a_{\lambda, r} = f^{\bar{\lambda}} \sum_{i=\abs{\bar{\lambda}}}^r \binom{i}{\abs{\bar{\lambda}}}\stirling{r}{i},
\end{equation}
where $\bar{\lambda}=(\lambda_2, \lambda_3, \ldots)$ with weight $\abs{\bar{\lambda}}$, $f^{\bar{\lambda}}$ is the number of standard Young tableaux of shape $\bar{\lambda}$, and $\stirling{r}{i}$ is a Stirling number of the second kind. 
\end{proposition}

\begin{proof}
Goupil and Chauve derived in \cite{GC06} the generating function
\begin{equation}\label{EqTensorGF}
\sum_{r \ge \abs{\bar{\lambda}}} a_{\lambda, r} \frac{x^r}{r!} = \frac{f^{\bar{\lambda}}}{\abs{\bar{\lambda}}!}e^{e^x-1}(e^x-1)^{\abs{\bar{\lambda}}}.
\end{equation}
By (24b) and (24f) in Chapter 1 of \cite{Sta97},
\begin{equation}\label{EqStanley}
\sum_{s \ge j} \stirling{s}{j} \frac{x^s}{s!} = \frac{(e^x-1)^j}{j!}
\end{equation} 
and
\begin{equation}
\sum_{t \ge 0} B_t \frac{x^t}{t!} = e^{e^x-1},
\end{equation}
where $B_0:=1$ and $B_t = \sum_{q=1}^t \stirling{t}{q}$ is the $t$-th Bell number, so we obtain from (\ref{EqTensorGF}) that
\begin{equation}
\frac{a_{\lambda, r}}{r!}= f^{\bar{\lambda}}\sum_{s+t=r} \frac{B_t}{s!t!}\stirling{s}{\abs{\bar{\lambda}}},
\end{equation}
and thus
\begin{equation}
\begin{split}
\frac{a_{\lambda, r}}{f^{\bar{\lambda}}} &= \sum_{t=0}^{r-\abs{\bar{\lambda}}} B_t\binom{r}{t}\stirling{r-t}{\abs{\bar{\lambda}}} \\ &= \stirling{r}{\abs{\bar{\lambda}}} + \sum_{t=1}^{r-\abs{\bar{\lambda}}} \sum_{q=1}^t \stirling{t}{q}\binom{r}{t}\stirling{r-t}{\abs{\bar{\lambda}}} \\ &= \stirling{r}{\abs{\bar{\lambda}}} + \sum_{q=1}^{r-\abs{\bar{\lambda}}} \sum_{t=q}^{r-\abs{\bar{\lambda}}} \stirling{t}{q}\binom{r}{t}\stirling{r-t}{\abs{\bar{\lambda}}}.
\end{split}
\end{equation}
By (24.1.3, II.A) of \cite{AS65},
\begin{equation}
\sum_{t=q}^{r-\abs{\bar{\lambda}}} \stirling{t}{q}\binom{r}{t}\stirling{r-t}{\abs{\bar{\lambda}}} = \binom{q+\abs{\bar{\lambda}}}{\abs{\bar{\lambda}}}\stirling{r}{q+\abs{\bar{\lambda}}},
\end{equation} 
so that
\begin{equation}
\begin{split}
\frac{a_{\lambda, r}}{f^{\bar{\lambda}}} &= \stirling{r}{\abs{\bar{\lambda}}} + \sum_{q=1}^{r-\abs{\bar{\lambda}}} \binom{q+\abs{\bar{\lambda}}}{\abs{\bar{\lambda}}}\stirling{r}{q+\abs{\bar{\lambda}}} \\ &= \stirling{r}{\abs{\bar{\lambda}}} + \sum_{i=\abs{\bar{\lambda}}+1}^{r} \binom{i}{\abs{\bar{\lambda}}}\stirling{r}{i} = \sum_{i=\abs{\bar{\lambda}}}^{r} \binom{i}{\abs{\bar{\lambda}}}\stirling{r}{i},
\end{split}
\end{equation}
as was to be shown.
\end{proof}

Now, let $b_{\lambda, r}$ be the multiplicities such that
\begin{equation}
(S^{(n-1, 1)})^{\otimes r} = \underset{\lambda \vdash n}{\bigoplus} b_{\lambda, r} S^{\lambda}.
\end{equation}
Goupil and Chauve also derived the generating function 
\begin{equation}\label{EqTensorGF2}
\sum_{r \ge \abs{\bar{\lambda}}} b_{\lambda, r} \frac{x^r}{r!} = \frac{f^{\bar{\lambda}}}{\abs{\bar{\lambda}}!}e^{e^x-x-1}(e^x-1)^{\abs{\bar{\lambda}}}, 
\end{equation}
so from Proposition \ref{PrIrreps} we can obtain a formula for the decomposition of $(S^{(n-1, 1)})^{\otimes r}$ as well.
\begin{corollary}\label{CoIrreps}
Let $\lambda \vdash n$ and $1 \le r \le n-\lambda_2$.  The multiplicity of $S^{\lambda}$ in the irreducible representation decomposition of $(S^{(n-1, 1)})^{\otimes r}$ is given by
\begin{equation}
b_{\lambda, r} = f^{\bar{\lambda}}\sum_{s=\abs{\bar{\lambda}}}^r (-1)^{r-s} \binom{r}{s}\left( \sum_{i=\abs{\bar{\lambda}}}^s \binom{i}{\abs{\bar{\lambda}}}\stirling{s}{i}\right).
\end{equation}
\end{corollary}  
\begin{proof}
Comparing (\ref{EqTensorGF2}) with (\ref{EqTensorGF}) gives
\begin{equation}
\sum_{r \ge \abs{\bar{\lambda}}} b_{\lambda, r} \frac{x^r}{r!} = \left(\sum_{s \ge \abs{\bar{\lambda}}} a_{\lambda, s} \frac{x^s}{s!}\right) e^{-x} = \left(\sum_{s \ge \abs{\bar{\lambda}}} a_{\lambda, s} \frac{x^s}{s!}\right) \left(\sum_{t \ge 0}\frac{(-x)^t}{t!}\right),
\end{equation}
so that 
\begin{equation}
\frac{b_{\lambda, r}}{r!}= \sum_{s+t=r} \frac{(-1)^t a_{\lambda, s}}{s!t!} = \sum_{s=\abs{\bar{\lambda}}}^r \frac{(-1)^{r-s} }{s!(r-s)!}\left(f^{\bar{\lambda}} \sum_{i=\abs{\bar{\lambda}}}^s \binom{i}{\abs{\bar{\lambda}}}\stirling{s}{i}\right),
\end{equation}
and the result follows.
\end{proof}
\begin{remark}
Corollary \ref{CoIrreps} is very similar to Proposition 2 of \cite{GC06}, but our result is cleaner, as it does not involve associated Stirling numbers of the second kind.  For another approach to the decomposition of tensor powers of $\varrho$, see \cite{Ful10}.
\end{remark}

\section{Connection to Markov Chain Mixing Time}

Consider the Markov chain on $S_n$ formed by first applying a random $n$-cycle to a deck of $n$ cards and then following with repeated random transpositions.  Formally, form a Markov chain $\{X_k\}$ on the symmetric group $S_n$ as follows: let $X_0$ be the identity, set $X_1=\pi X_0$, where $\pi$ is a uniformly selected $n$-cycle, and for $k \ge 2$ set $X_k = \tau_k X_{k-1}$, where $\tau_k$ is a uniformly selected transposition.  Observe that $X_k \in A_n$ when $n$ and $k$ are of the same parity.  Otherwise, $X_k \in S_n \backslash A_n$.  Let $\mu_k$ be the law of $X_k$, and let $U_k$ be the uniform measure on $A_n$ if $X_k \in A_n$ and the uniform measure on $S_n \backslash A_n$ if $X_k \in S_n \backslash A_n$.  What is the total variation distance between $\mu_k$ and $U_k$? 

The goal of this section is to prove the following:

\begin{theorem}\label{ThMain}
For any $c>0$, after one $n$-cycle and $cn$ transpositions, 
\begin{equation}
\frac{e^{-2c}}{e} - o(1) \le \|\mu_{cn+1} - U_{cn+1}\|_{\emph{TV}} \le \frac{e^{-2c}}{2\sqrt{1-e^{-4c}}} + o(1)
\end{equation}
as $n$ goes to infinity.
\end{theorem}

The upper bound follows from the approach of \cite{DS81}.  For the (lazy) random transposition shuffle on $n$ cards, the time-homogeneous chain on $S_n$ with increment measure $\upsilon$ that assigns mass $ \frac{1}{n}$ to the identity and $\frac{2}{n^2}$ to each of the $\frac{n(n-1)}{2}$ transpositions $\tau$, Diaconis and Shahshahani derived the bound
\begin{equation}\label{EqSpectral}
4\|\mu_k-U \|_{\text{TV}}^2 \le \sum_{\substack{\rho \in \widehat{S_n}\\ \rho \neq \rho_{\text{triv}}}} d_{\rho}^2 \left(\frac{1}{n} + \frac{(n-1)\chi_{\rho}(\tau)}{nd_{\rho}}\right)^{2k},
\end{equation}
where $U$ is the uniform measure on $S_n$, $\widehat{S_n}$ is the set of irreps of $S_n$, and $d_{\rho}$ and $\chi_{\rho}(\tau)$ denote the dimension and the character at $\tau$ of the representation $\rho$, respectively.  Careful computations of the terms on the RHS of (\ref{EqSpectral}) gave a mixing time of $O(n \ln n)$, and explicit constants were later calculated by Saloff-Coste and Z\'u\~niga in \cite{S-CZ08}.

Inequality (\ref{EqSpectral}) comes from the theory of non-commutative Fourier analysis on $S_n$.  It carries the following routine extension (carefully spelled out in Chapter $2$ of \cite{Din14}) to the $n$-cycle-to-transpositions chain:

\begin{equation}\label{EqFinalUB}
4\|\mu_{k+1} - U_{k+1}\|_{\text{TV}}^2 \le \frac{1}{2} \sum_{\substack{\rho \in \widehat{S_n}\\ \rho \neq \rho_{\text{triv}}, \rho_{\text{sign}}}} d_{\rho}^2\left(\frac{\chi_{\rho}(\tau)}{d_{\rho}}\right)^{2k}\left(\frac{\chi_{\rho}(\pi)}{d_{\rho}}\right)^2 .
\end{equation}

\begin{proposition}\label{ThUBn}
For any $c>0$, after one $n$-cycle and $cn$ transpositions, 
\begin{equation}
4\|\mu_{cn+1} - U_{cn+1}\|_{\emph{TV}}^2 \le \frac{e^{-4c}}{1-e^{-4c}} + o(1)
\end{equation}
as $n$ goes to infinity.
\end{proposition} 

\begin{proof}Let $\chi^{\lambda}_{\gamma}$ denote the character of $S^{\lambda}$ on the cycle type ${\gamma}$.  The first and most critical step of the proof is the observation that, discounting $(n)$ and $(1^n)$, $\chi^{\lambda}_{(n)} = 0$ for all $\lambda$ except the hook-shaped ones, for which $\lambda_2=1$. This is an almost trivial consequence of the Murnaghan-Nakayama rule, as it is impossible to remove a rim hook of size $n$ from a Young diagram of size $n$ unless the Young diagram itself is the rim hook.  Moreover, for a hook-shaped $\lambda$, it is clear that $\chi^{\lambda}_{(n)}$ is equal to $1$ if $\lambda$ has an odd number of rows and $-1$ if $\lambda$ has an even number of rows.  Thus we arrive at a significant simplication of (\ref{EqFinalUB}), namely that
\begin{equation}
4\|\mu_{k+1} - U_{k+1}\|_{\text{TV}}^2 \le \frac{1}{2} \sum_{\lambda \in \Lambda_n} \left(\frac{\chi^{\lambda}_{(2, 1^{n-2})}}{\dim S^{\lambda}}\right)^{2k},
\end{equation}
where 
\begin{equation}
\Lambda_n = \{\lambda \vdash n: \lambda_1 > 1 \text{ and }\lambda_2 = 1\}.
\end{equation}

The normalized characters $\frac{\chi^{\lambda}_{(2, 1^{n-2})}}{\dim S^{\lambda}}$ have a simple description when $\lambda \in \Lambda_n$: let $j$ be one less than the number of rows of $\lambda$, then for $1 \le j \le \left\lfloor \frac{n-1}{2} \right\rfloor$, 
\begin{equation}\label{PrUBnNC}
\frac{\chi^{(n-j, 1^j)}_{(2, 1^{n-2})}}{\dim S^{(n-j, 1^j)}} = \frac{n-1-2j}{n-1}.
\end{equation}
This is a special case of the identity
\begin{equation}
\frac{\chi^{\lambda}_{(2, 1^{n-2})}}{\dim S^{\lambda}} = \frac{\sum_i (\lambda_i^2-(2i-1)\lambda_i)}{n(n-1)}, 
\end{equation} 
known as early as to Frobenius in \cite{Fro00}.

Fix any $c>0$.  By calculus, for $n-1-2j > 0$,  
\begin{equation}\label{EqUBnLh}
\lim_{n\rightarrow \infty} \left(\frac{n-1-2j}{n-1}\right)^{2cn} = e^{-4cj}.
\end{equation}
Thus (\ref{PrUBnNC}) and the fact that $\chi^{\lambda}_{\gamma} = \pm \chi^{\lambda'}_{\gamma}$, where $\lambda'$ is the conjugate partition of $\lambda$ (see p. 25 of \cite{Jam78}), imply that
\begin{equation}
\sum_{\lambda \in \Lambda_n} \left(\frac{\chi^{\lambda}_{(2, 1^{n-2})}}{\dim S^{\lambda}}\right)^{2cn} \sim \begin{cases}
		2\sum\limits_{j=1}^{(n-2)/2} e^{-4cj} &\hspace{0.15in}\text{$n$ is even}\\
		2\sum\limits_{j=1}^{(n-3)/2} e^{-4cj} &\hspace{0.15in}\text{$n$ is odd.}
    \end{cases}
\end{equation}
Summing the geometric series gives
\begin{equation}
4\|\mu_{cn+1} - U_{cn+1}\|_{\text{TV}}^2 \le \frac{1}{2} \sum_{\lambda \in \Lambda_n} \left(\frac{\chi^{\lambda}_{(2, 1^{n-2})}}{\dim S^{\lambda}}\right)^{2cn} \sim \frac{e^{-4c}}{1-e^{-4c}},
\end{equation}
as was to be shown.
\end{proof}

For measures $\mu$ and $\nu$ on a set $G$, a classic approach to finding a lower bound for $\|\mu-\nu\|_{\text{TV}}$ is to identify a subset $A$ of $G$ where $\abs{\mu(A)-\nu(A)}$ is close to maximal.  In many mixing problems involving the symmetric group, it is convenient to make $A$ either the set of fixed-point-free permutations or its complement, since it is well-known that the distribution of the number of fixed points with respect to the uniform measure on $S_n$ is asymptotically $\mathcal{P}(1)$, the Poisson distribution of mean one.  The same is true for the distribution of fixed points with respect to the uniform measure on either $A_n$ or $S_n \backslash A_n$.  See Theorem $4.3.3$ of \cite{Din14} for a proof. 

For Diaconis and Shahshahani's random transposition shuffle, $A$ is the set of permutations with one or more fixed points, and finding $\mu_k(A)$ boils down to a coupon collector's problem.  Let $B$ be the event that, after $k$ transpositions, at least one card is untouched.  It is not difficult to see that $\mu_k(A) \ge \mathbf{P}(B)$, where $\mathbf{P}(B)$ is equal to the probability that at least one of $n$ coupons is still missing after $2k$ trials.  The coupon collector's problem is well-studied, so this immediately gives a lower bound for $\mu_k(A)$, which in turn produces a lower bound for $\|\mu_k(A)-U(A)\|_{\text{TV}}$.  

The above argument is so short and simple that it was tagged onto the end of the introduction of \cite{DS81}, as if an afterthought.  Unfortunately, it is inapplicable to our problem, since the initial $n$-cycle obliterates the core of the argument.  Instead, we will fully characterize the distribution of $\chi_{\varrho}$ with respect to $\mu_{k+1}$ by deriving all moments of $\chi_{\varrho}$ with respect to $\mu_{k+1}$.  Let $E_{\mu}$ denote expectation with respect to $\mu$, then as observed in Chapter 3D of \cite{Dia88}, 
\begin{equation}\label{EqDia}
E_{\mu}(\chi_{\rho})=\sum_{\sigma\in S_n}\mu(\sigma)\text{tr}(\rho(\sigma)) = \text{tr}\left(\sum_{\sigma\in S_n}\mu(\sigma)\rho(\sigma)\right) = \text{tr}(\hat{\mu}(\rho)), 
\end{equation}  
so that
\begin{equation}
E_{\mu}((\chi_{\varrho})^r)=\sum_{\lambda\vdash n} a_{\lambda, r} \text{tr}(\hat{\mu}(S^{\lambda})),
\end{equation}
where $\hat{\mu}$ is the Fourier transform of $\mu$ and 
\begin{equation}\label{EqAlready}
\text{tr}(\widehat{\mu_{k+1}}(S^{\lambda})) = \chi^{\lambda}_{(n)}\left(\frac{\chi^{\lambda}_{(2, 1^{n-2})}}{\dim S^{\lambda}}\right)^k.
\end{equation}

\begin{proposition}\label{ThLimitn}
Fix any $c>0$.  As $n$ approaches infinity, the distribution of the number of fixed points after one $n$-cycle and $cn$ transpositions converges to $\mathcal{P}(1-e^{-2c})$.
\end{proposition}
\begin{proof}
One can deduce from the moment-generating function that the $r$-th moment of $\mathcal{P}(\nu)$ is $\sum_{i=1}^r  \stirling{r}{i}{\nu}^i$.  It is a standard result that $\widehat{\mu_{cn+1}}(S^{(n)})=1$, and we will ignore the alternating representation because it suffices to consider the first $n-2$ moments, in which the alternating representation does not appear.  For the non-trivial and non-alternating representations, we take advantage of previous computations and synthesize (\ref{PrUBnNC}), (\ref{EqUBnLh}) with $n$ instead of $2n$, and (\ref{EqAlready}) to obtain
\begin{equation}\label{EqSim}
\widehat{\mu_{cn+1}}(S^{\lambda}) \sim \begin{cases}
		(-1)^{\abs{\bar{\lambda}}}e^{-2c\abs{\bar{\lambda}}} \hspace{0.05in} &\lambda \in \Lambda_n\\
		0 &\text{otherwise}.
   \end{cases}
\end{equation}

By Proposition \ref{PrIrreps} (second line below) and (\ref{EqSim}) (fourth line), for $1 \le r \le n-2$, 
\begin{equation}
\begin{split}
E_{\mu_{cn+1}}((\chi_{\varrho})^r) &= a_{(n), r} + \sum_{\lambda \in \Lambda_n} a_{\lambda, r} \widehat{\mu_{cn+1}}(S^{\lambda}) \\ &= \sum_{i=1}^r \stirling{r}{i} + \sum_{\abs{\bar{\lambda}}=1}^{n-2} \sum_{i=\abs{\bar{\lambda}}}^r \stirling{r}{i}\binom{i}{\abs{\bar{\lambda}}}\widehat{\mu_{cn+1}}(S^{\lambda}) \\ &= \sum_{i=1}^r \stirling{r}{i} + \sum_{i=1}^r\sum_{\abs{\bar{\lambda}}=1}^{i} \stirling{r}{i}\binom{i}{\abs{\bar{\lambda}}}\widehat{\mu_{cn+1}}(S^{\lambda}) \\ &\sim \sum_{i=1}^r \stirling{r}{i} + \sum_{i=1}^r\sum_{\abs{\bar{\lambda}}=1}^{i}  \stirling{r}{i}\binom{i}{\abs{\bar{\lambda}}} (-e^{-2c})^{\abs{\bar{\lambda}}} \\ &=\sum_{i=1}^r \stirling{r}{i}\left(1+\sum_{\abs{\bar{\lambda}}=1}^{i} \binom{i}{\abs{\bar{\lambda}}} (-e^{-2c})^{\abs{\bar{\lambda}}}\right) \\ &=\sum_{i=1}^r \stirling{r}{i}(1-e^{-2c})^i.
\end{split}
\end{equation}
This shows that the first $n-2$ moments of $\chi_{\varrho}$ with respect to $\mu_{cn+1}$ approach those of $\mathcal{P}(1-e^{-2c})$, and convergence follows from the method of moments.
\end{proof}
\begin{corollary}\label{CoLBn}
For any $c>0$, after one $n$-cycle and $cn$ transpositions, 
\begin{equation}
\|\mu_{cn+1} - U_{cn+1}\|_{\emph{TV}} \ge \frac{e^{-2c}}{e} - o(1)
\end{equation}
as $n$ goes to infinity.
\end{corollary}
\begin{proof}
Let $A$ be the set of fixed-point-free permutations.  Then
\begin{equation}
\begin{split}
\|\mu_{cn+1} - U_{cn+1}\|_{\text{TV}} &\ge \abs{\mu_{cn+1}(A)-U_{cn+1}(A)} \\ &\sim e^{e^{-2c}-1} - \frac{1}{e} = \frac{1}{e}\left(e^{-2c}+\frac{(e^{-2c})^2}{2!} + \cdots \right) \ge \frac{e^{-2c}}{e},
\end{split}
\end{equation}
as was to be shown.
\end{proof}

Together with Proposition \ref{ThLimitn}, Corollary \ref{CoLBn} completes the proof of Theorem \ref{ThMain}.


\begin{thebibliography}{99}

\bibitem{AS65} M. Abramowitz and I. A. Stegun, eds., \emph{Handbook of Mathematical Functions with Formulas, Graphs, and Mathematical Tables}, Dover, New York, 1965.

\bibitem{Dia88} P. Diaconis, \emph{Group Representations in Probability and Statistics}, IMS Lecture Notes Monogr. Ser. 11, Inst. Math. Statist., Hayward, CA, 1988. 

\bibitem{DS81} P. Diaconis and M. Shahshahani, \emph{Generating a random permutation with random transpositions}, Z. Wahrsch. verw. Geb. \textbf{57} (1981), no. 2, 159-179.

\bibitem{Din14} S. Ding, \emph{A Random Walk in Representations}, Ph.D. Thesis, University of Pennsylvania, 2014.

\bibitem{Fro00} F. G. Frobenius, \emph{\"Uber die Charaktere der symmetrischen Gruppen}, Sitz. Konig. Preuss. Akad. Wissen. (1900), 516-534.

\bibitem{Ful10} J. Fulman, \emph{Separation cutoffs for random walk on irreducible representations}, Ann. Comb. \textbf{14} (2010), no. 3, 319-337.

\bibitem{FH91} W. Fulton and J. Harris, \emph{Representation Theory: A First Course}, GTM 129, Springer-Verlag, New York, 1991.

\bibitem{GC06} A. Goupil and C. Chauve, \emph{Combinatorial operators for Kronecker powers of representations of $S_n$}, S\'eminaire Lotharingien de Combinatoire, \textbf{54} (2006), B54j.

\bibitem{Jam78} G. D. James, \emph{The Representation Theory of the Symmetric Groups}, LNM 682, Springer-Verlag, Berlin, 1978.

\bibitem{S-CZ08} L. Saloff-Coste and J. Z\'u\~niga, \emph{Refined estimates for some basic random walks on the symmetric and alternating groups}, ALEA Lat. Am. J. Probab. Math. Stat. \textbf{4} (2008), 359-392.

\bibitem{Sta97} R. P. Stanley, \emph{Enumerative Combinatorics, Vol. I}, Wadsworth, Monterey, CA, 1986, Cambridge Stud. Adv. Math. 49, reprinted by Cambridge Univ. Press, Cambridge, 1997.

\end{thebibliography}
\end{document}